\documentclass[11pt]{amsart}

\usepackage{amsmath, amssymb, amsfonts, amsthm, bm}
\usepackage{graphicx}
\usepackage{enumerate}

\makeatletter
\@namedef{subjclassname@2010}{%
  \textup{2010} Mathematics Subject Classification}
\makeatother

\newtheorem{thm}{Theorem}[section]
\newtheorem{cor}[thm]{Corollary}
\newtheorem{lem}[thm]{Lemma}
\newtheorem{proposition}[thm]{Proposition}
\newtheorem{prob}[thm]{Problem}

\theoremstyle{definition}

\newtheorem{rem}[thm]{Remark}

\numberwithin{equation}{section}

\DeclareMathOperator{\conv}{conv}
\DeclareMathOperator{\diam}{diam}

\begin{document}


\baselineskip=17pt



\title[Around the Petty theorem on equilateral sets]{Around the Petty theorem on equilateral sets}

\author[T. Kobos]{Tomasz Kobos}
\address{Faculty of Mathematics and Computer Science \\ Jagiellonian University \\ Lojasiewicza 6, 30-348 Krakow, Poland}
\email{Tomasz.Kobos@im.uj.edu.pl}

\date{}

\begin{abstract}
The main goal of this paper is to provide an alternative proof of the following theorem of Petty: in the normed space of dimension at least three, every $3$-element equilateral set can be extended to a $4$-element equilateral set. Our approach is based on the result of Kramer and N\'emeth about inscribing a simplex into a convex body. To prove the theorem of Petty, we shall also establish that for every $3$ points in the normed plane, forming an equilateral set of the common distance $p$, there exists a fourth point, which is equidistant to the given points with the distance not larger than $p$. We will also improve the example given by Petty and obtain the existence of a smooth and strictly convex norm in $\mathbb{R}^n$, which contain a maximal $4$-element equilateral set. This shows that the theorem of Petty cannot be generalized to higher dimensions, even for smooth and strictly convex norms.
\end{abstract}



\maketitle

\section{Introduction}
Let $X$ be a real $n$-dimensional vector space equipped with a norm $|| \cdot ||$. We say that a set $S \in X$ is \emph{equilateral}, if there is a $p>0$ such that $||x-y||=p$ for all $x, y \in S, x \neq y$. In such situation we will say that $S$ is a $p$-equilateral set. By $e(X)$ let us denote the \emph{equilateral dimension} of the space $X$, defined as the maximal cardinality of a equilateral set in $X$. For many classic spaces (like $\ell_{p}^n$) determining the equilateral dimension is an open problem. It is not difficult to show that the equilateral dimension of $n$-dimensional space equipped with the Euclidean norm is equal to $n+1$, and it is equal to $2^n$ for the $\ell_{\infty}$ norm. It is known (see \cite{petty} and \cite{soltan}) that $2^n$ is, in fact, an upper bound for the equilateral dimension of any normed space $X$ of dimension $n$. Moreover, the equality holds if, and only if there exists a linear isometry between $X$ and $\ell_{\infty}^{n}$. 

It is believed that $n+1$ is similarly a lower bound for the equilateral dimension of any $n$-dimensional normed space. For $n=1$ and $n=2$ it is an easy exercise. For $n=3$ it has been proved by Petty and in the case $n=4$ quite recently by Makeev in \cite{makeev2}. For $n \geq 5$ only weaker estimates on the size of a maximal equilateral set are known (for the best bound at this moment see \cite{swanepoel2}). In the three dimensional setting even more can be demonstrated. We have the following
\begin{thm}[\bf{Petty} \cite{petty}]
\label{tw1}
Let $X$ be a real $3$-dimensional vector space, equipped with a norm $|| \cdot ||$. Assume that $a, b, c \in X$ form a $p$-equilateral set in the norm $|| \cdot ||$. Then, there exists $d \in X$ such that
$$||d-a||=||d-b||=||d-c||=p.$$
In other words, every equilateral set of $3$ elements can be extended to an equilateral set of $4$ elements.
\end{thm}

The main goal of this paper is to give an alternative proof of the theorem of Petty, which is obtained in section \ref{sec:5}. In his original reasoning Petty used the two dimensional result called monotonicity lemma (see section $3.5$ of \cite{geometry}). Our approach we will be based on the Theorem \ref{nemeth} of Kramer and N\'emeth. We cover the required material in section \ref{sec:prel}.

To take advantage of the mentioned techniques, we will need to investigate the properties of the circumcircle of an equilateral set in the plane. We shall prove
\begin{proposition}
\label{tw2}
Let $|| \cdot ||$ be a norm on the plane. Assume that $a, b, c \in \mathbb{R}^2$ form a $p$-equilateral set in this norm. Then, there exists $s \in \mathbb{R}^2$ such that
$$||a-s||=||b-s||=||c-s|| \leq p.$$
In other words, on the plane every equilateral set of size $3$ has a circumcircle of a radius not greater then the common distance.
\end{proposition}
Such a result have also appeared in \cite{martini} (see Lemma $2.4$ and Theorem $3.1$) in the case of strictly convex norm, but we give a simple proof.

In section \ref{sec:example}, we study an extension of the example given by Petty in \cite{petty}. For a given $n \geq 4$, he has constructed a norm in the space $\mathbb{R}^n$, with the $4$-element equilateral set which cannot be extended to a $5$-element equilateral set (in other words, it is maximal with respect to inclusion). In particular, Theorem \ref{tw1} cannot be generalized to higher dimensions. However, the norm given by Petty is not smooth and not strictly convex. It is therefore natural to ask, whether the smoothness or strict convexity of the norm would enable us to generalize Theorem \ref{tw1} to higher dimensions. We answer this question negatively in
\begin{thm}
\label{tw3}
For every $n \geq 4$ there exist a smooth and strictly convex norm $|| \cdot ||$ in $\mathbb{R}^n$ and $a_1, a_2, a_3, a_4 \in \mathbb{R}^n$ forming a $p$-equilateral set, but there is no $a_5$ such that 
$$||a_5 - a_1||=||a_5 - a_2||=||a_5 - a_3||=||a_5 - a_4||=p.$$
\end{thm}
This improves the example given by Petty.

For a survey on equilateral sets in finite dimensional normed spaces, see \cite{swanepoel1}.

\section{Preliminaries}
\label{sec:prel}
We shall recall some standard definitions of Banach space theory and convex geometry. A convex and compact set $C \subset \mathbb{R}^n$ is called a \emph{convex body}, if its interior is non-empty. Convex body $C$ is \emph{symmetric}, if it is symmetric with respect to the origin. It is \emph{smooth}, if every point on the boundary of $C$ lies on exactly one supporting hyperplane of $C$. Convex body $C$ is a \emph{strictly convex body}, if it does not contain a segment on the boundary. The unit ball of any norm in $\mathbb{R}^n$ is a symmetric convex body and vice versa -- every symmetric convex body is a unit ball of exactly one norm. We say that a norm $|| \cdot ||$ is smooth or strictly convex, if the unit ball of norm $|| \cdot ||$ is smooth or strictly convex, respectively.  A sphere in the given norm on the plane is called a $\emph{circle}$.

A set $P \subset \mathbb{R}^n$ is called a \emph{convex polytope}, if it is the convex hull of a finite number of points, i.e. $P = \conv \{ p_1, p_2, \ldots, p_m \}$. If there does not exist a proper subset $S$ of $\{ p_1, p_2, \ldots, p_m \}$ such that $P = \conv S$, then the points $p_i$ are \emph{vertices} of the convex polytope $P$. Convex polytope $P \subset \mathbb{R}^n$ with exactly $n+1$ vertices is called a \emph{simplex}. A simplex is \emph{non-degenerated}, if it is not contained in the affine subspace of dimension $n-1$. A (positive) \emph{homothet} of a set $A \subset \mathbb{R}^n$ is a set of the form $\lambda A + v = \{ \lambda a + v : a \in A\}$, where $\lambda > 0$ and $v \in \mathbb{R}^n$ are arbitrary.

In the latter considerations we shall use the following theorem of Kramer and N\'emeth, which give a sufficient conditions that a convex body $C \subset \mathbb{R}^n$ has to satisfy in order to have the following property: every non-degenerated simplex of the space $\mathbb{R}^n$ has a homothet inscribed in $C$.

\begin{thm}[\bf{Kramer \& N\'emeth \cite{kramer}}]
\label{nemeth}
Let $C$ be smooth and strictly convex body of the space $\mathbb{R}^n$ and let $p_0, p_1, \ldots, p_n$ be vertices of a non-degenerated simplex. Then, there exist $z \in \mathbb{R}^n$ and $r>0$, such that points $z+rp_0, z + rp_1, \ldots, z+rp_n$ lie on the boundary of $C$. 
\end{thm}
It is worth pointing out that, while smoothness is necessary, the condition of strict convexity can in fact be dropped, as Gromov in \cite{gromov} and Makeev in \cite{makeev1} have shown independently later on. For our purposes however the weaker version of the theorem is sufficient.

It is much harder to guarantee the uniqueness of such inscribed homothet. However, in the two dimensional case we have the following
\begin{proposition}
\label{jedynosc}
Let $C$ be a strictly convex body in the plane and let $p_0, p_1, p_2$ be vertices of the non-degenerated triangle. Then there exists at most one pair of $z \in \mathbb{R}^2$ and $r>0$, such that points $z+rp_0, z + rp_1, z+rp_2$ lie on the boundary of $C$. 
\end{proposition}
\begin{proof} An elementary proof can be found in \cite{geometry}, section 3.2. \end{proof}

Combining these two results, with $C$ chosen as a smooth and strictly convex body in the plane, we obtain
\begin{cor}
\label{wniosek}
Let $C$ be a smooth and strictly convex norm in the plane. Then, for every non-degenerated triangle, there exists exactly one homothet of $C$ passing through its vertices.
\end{cor}

\begin{rem}
A generalization of the Proposition \ref{jedynosc} to higher dimensions does not hold in the following strong sense: if every simplex in $\mathbb{R}^n$ (where $n \geq 3$) has at most one homothet inscribed into a fixed convex body $C$, then $C$ must be an ellipsoid. This characterization of finite dimensional Hilbert space was established by Goodey in \cite{goodey}.
\end{rem}

To take advantage of the preceeding results we have to assume smoothness and strict convexity of a norm. To reduce the general case to such setting, we will use a special kind of smooth and strictly convex approximation. Proof of Theorem \ref{tw3} will also rely on this technique. We have the intuitive
\begin{proposition}
\label{aproksymacja}
Assume that in the space $\mathbb{R}^n$ there are given: a norm $|| \cdot ||$ and the unit vectors $\pm p_1, \pm p_2, \ldots, \pm p_m$, which are also vertices of a symmetric convex polytope. Then for every $\varepsilon > 0$ there exists a smooth and strictly convex norm $|| \cdot ||_0$ in $\mathbb{R}^n$ such that
$$(1-\varepsilon)||x||_0 \leq ||x|| \leq (1+\varepsilon)||x||_0$$
for every $x \in \mathbb{R}^n$ and $||p_i||_0=1$ for $i=1, 2, \ldots, m$.
\end{proposition}
\begin{proof} See \cite{ghomi} for a much more general result. \end{proof}

\section{Circumcircle of the equilateral set on the plane.}
\label{sec:plane}
In this section we prove \ref{tw2}. Results from the preceeding section are not required here. We begin with a simple
\begin{lem}
\label{lem2}
Let $|| \cdot ||$ be a norm in the plane and let $v \in \mathbb{R}^2$ be a non-zero vector. Consider a mapping $f$, defined on the unit disc $B$ of norm $|| \cdot ||$ as
$$f(x) = \max \{ t \geq 0: ||x+tv|| = 1 \}.$$
Then $f$ is continuous and bounded by $\frac{2}{||v||}$.
\end{lem}
\begin{proof} Indeed, it is easy to see that $f$ is bounded by $\frac{2}{||v||}$. In fact, for any point $x \in B$ of unit disc and any $t > \frac{2}{||v||}$, we have
$$|||x+tv|| \geq ||tv|| - ||x|| > 2 - 1 = 1,$$ 
from the triangle inequality. This proves the second part.

It is sufficient to prove the continuity of $f$. Let $D$ be the diameter of the $B$ orthogonal to $v$ and denote by $P: B \to D$ the orthogonal projection. It is clear that $f$ factors as $f = f \circ P$ and it is therefore enough to check the continuity of the restriction $\tilde{f}=f|_{D}$.

Let $(x_n)_{n \in \mathbb{N}} \subset D$ be a sequence converging to $x \in D$. As $\tilde{f}$ is bounded, it is enough to check that every convergent subsequence of $(\tilde{f}(x_n))_{n \in \mathbb{N}}$ converges to $\tilde{f}(x)$. So, let us assume that $\lim_{k \to \infty} \tilde{f}(x_{n_k}) = y$ for some subsequence $(\tilde{f}(x_{n_k}))_{k \in \mathbb{N}}$. Since
$$||x+yv|| = \lim_{k \to \infty} ||x_{n_k} + \tilde{f}(x_{n_k})v|| = 1,$$
we have $y \leq \tilde{f}(x)$, by the definition of $f$. 

On the other hand, it is immediate to check that $\tilde{f}$ is a concave function. As concave functions defined on the closed intervals are lower semi-continuous we must have $y \geq \tilde{f}(x)$. Therefore $y=\tilde{f}(x)$ and the lemma is proved. \end{proof}

\emph{Proof of Proposition \ref{tw2}.} We can suppose that $p=1$ and denote the unit circle in norm $|| \cdot ||$ by $S$. Let us define mappings $f, g: S \rightarrow \mathbb{R}$ as
$$f(x) = \max \{t \geq 0:  ||x + t(c-a)|| = 1 \} \text{ and } g(x) = \max \{t \geq 0:  ||x + t(b-a)|| = 1 \}.$$
By the preceeding lemma, mappings $f$ and $g$ are continuous and bounded by $2$. At the same time, we see that $f(a-c) = g(a-b) = 2$. Therefore $f(x) \geq g(x)$ for $x=a-c$ and $f(x) \leq g(x)$ for $x=a-b$.

The intermediate value theorem implies that on the closed arc of unit circle between points $a-b$ and $a-c$, there exists point $z$ such that $f(z)=g(z)=r$ for some $r \geq 0$ (see Figure \ref{fig:1}). For such a chosen $z$ we have
$$||z+r(c-a)||=||z+r(b-a)||=1.$$ 
If $r>0$, then the points $z, z+r(c-a), z+r(b-a)$ are vertices of non-degenerated triangle inscribed into the unit circle, which is a homothet of the triangle with vertices $a$, $b$, $c$. Therefore, we can circumscribe a circle of a radius $\frac{1}{r}$ on the triangle with vertices $a, b, c$. It suffices to show that $r \geq 1$ (and in particular $r>0$).

\begin{figure}
\centering
\includegraphics{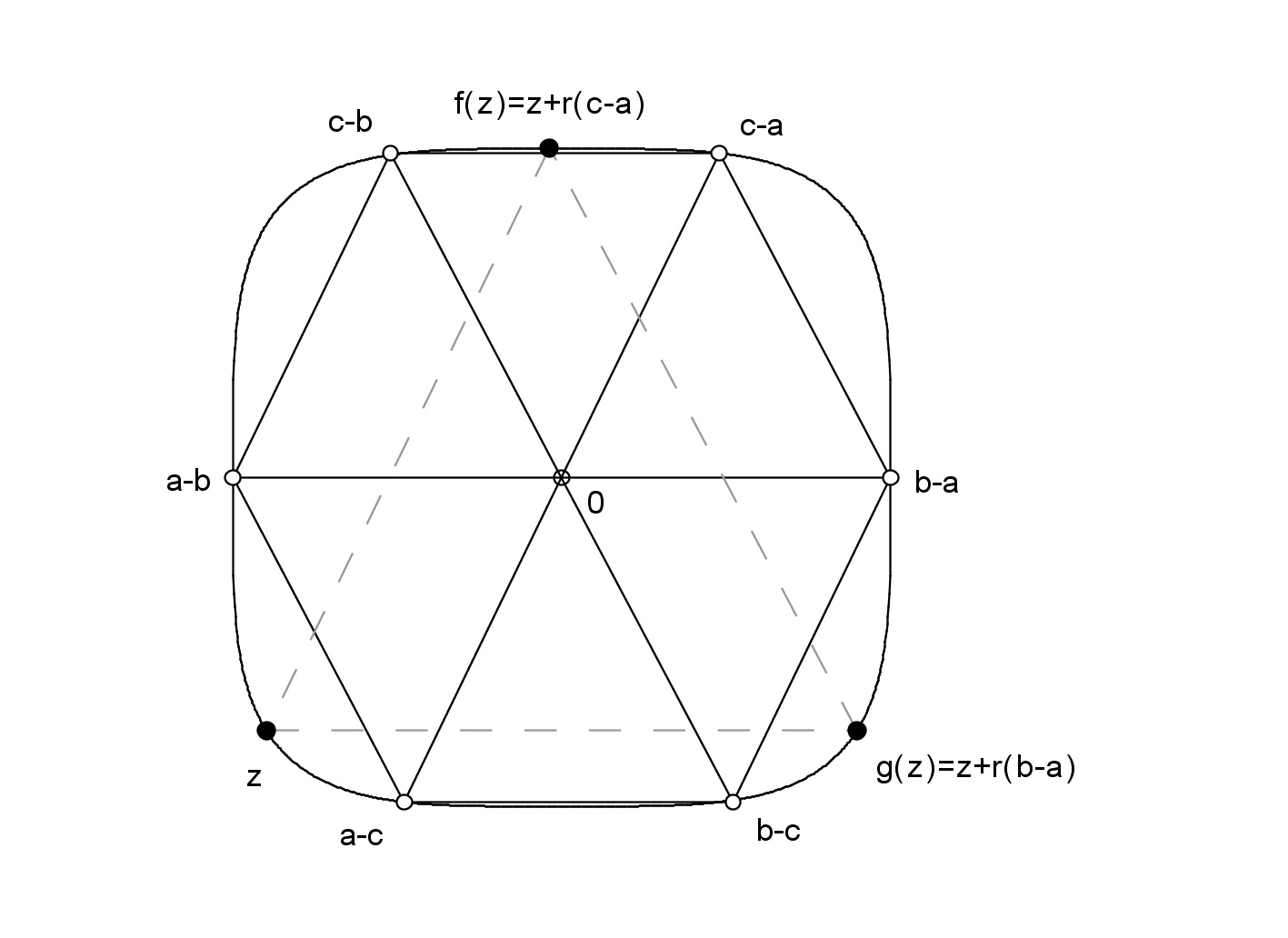}
\caption{Proof of Proposition \ref{tw2}}
\label{fig:1} 
\end{figure}

To do that, let us consider the parallelogram with vertices $0, c-b, a-b, a-c$. Points $z$ and $z+r(c-a)$ lie outside of this parallelogram and the segment connecting them is parallel to one of its sides (as $(c-b) - (a-b) = c-a$). Hence $r \geq 1$. \qed

\section{Proof of Petty's theorem}
\label{sec:5}
In this section we prove Theorem \ref{tw1}. In the sequel we need only the two and three dimensional cases of the following lemma, but we prove it in general form.
\begin{lem}
\label{lem3}
In the space $\mathbb{R}^n$ we are given the points $p_0, p_1, \ldots, p_n$, which are vertices of the non-degenerated simplex. Then the points $p_i - p_j$, where $0 \leq i \neq j \leq n$, are vertices of the convex polytope $\conv \{ p_i - p_j: 0 \leq i \neq j \leq n \}$.
\end{lem}
\begin{proof} For the sake of a contradiction let us suppose that some point can be written as a convex combination of the remaining points. Without loss of generality, we can assume that it is $p_n - p_0$.

Let us denote $q_i = p_i - p_0$ for $i=0, 1, 2, \ldots, n$. Because the simplex with the vertices $p_0, p_1, \ldots, p_n$ is non-degenerated, the vectors $q_i$ are linearly independent for $i=1, 2, \ldots, n$. We will show, that if
$$q_n = \sum_{0 \leq k \neq l \leq n} t_{k, l} (p_k - p_l),$$
where $t_{k, l} \in [0, 1]$ and $\sum_{0 \leq k \neq l \leq n} t_{k, l} = 1$, then $t_{n, 0} = 1$ and $t_{k, l} = 0$ for $(k, l) \neq (n, 0)$. We have
$$q_n = \sum_{0 \leq k \neq l \leq n} t_{k, l} (p_k - p_l) = \sum_{0 \leq k \neq l \leq n} t_{k, l} (p_k - p_0 + p_0 - p_l) = \sum_{0 \leq k \neq l \leq n} t_{k, l} (q_k - q_l)$$
$$= \sum_{k=1}^n  \left( \sum_{\substack{0 \leq l \leq n \\ l \neq k}} t_{k, l} - \sum_{\substack{0 \leq l \leq n \\ l \neq k}}^n t_{l, k} \right ) q_k.$$
The vectors $q_1, q_2, \ldots, q_n$ are linearly independent and hence, comparing the coefficients at $q_n$ gives us 
$$\sum_{l=0}^{n-1} t_{n, l} - \sum_{l=0}^{n-1} t_{l, n} = 1.$$
Because the numbers $t_{k, l}$ are non-negative and their sum is $1$, we obtain that 
$$\sum_{l=0}^{n-1} t_{n, l} = 1 \text{ and } t_{k, l} = 0 \text{ for } k \neq n.$$
If $t_{n, 0}=1$ and $t_{n, l} = 0$ for $1 \leq l \leq n-1$ the claim follows. So, let us assume that $t_{n, l} \neq 0$ for some $1 \leq l \leq n-1$. But then it is easy to see that the coefficient at $q_l$ in the considered combination is equal to $-t_{n, l}$ and therefore it is non-zero. This is a contradiction with the linear independence of $q_1, q_2, \ldots, q_n$, which proves the lemma. \end{proof}

With the results of preceeding sections at hand, we are ready to give an alternative proof of Petty's theorem.

\paragraph{Proof of Theorem \ref{tw1}}. Suppose that $p=1$. First we assume that the norm $|| \cdot ||$ is smooth and strictly convex. Let $\pi$ be the affine plane containing $a, b, c$. In a similar fashion to Makeev in \cite{makeev2}, we will consider sections of the unit ball $B$ with planes parallel to $\pi$ and base our approach on the continuity of these cuts. All such sections, which are not empty and not single-point are smooth and strictly convex bodies in the plane. From Corollary \ref{wniosek} we know that in all such sections we can inscribe exactly one homothet of the triangle with vertices $a, b, c$. To be precise, let $v$ be the unit vector perpendicular to $\pi$ and denote $B_t = B \cap (\pi + tv)$. Let $t_1<t_2$ be such that
$$\# B_t > 1 \iff t \in (t_1,t_2).$$
Corollary \ref{wniosek} yields that for every $t \in (t_1, t_2)$ there exists exactly one pair of $z \in \pi$ and $r > 0$ such that
$$||z+ra+tv||=||z+rb+tv||=||z+rc+tv||=1.$$
Denote by $z: (t_1,t_2) \rightarrow \pi$ and $r: (t_1,t_2) \rightarrow (0, \infty)$ mappings which assign to a given $t \in (t_1, t_2)$ the unique mentioned vector and unique positive number, respectively. We shall verify, that the mapping $(z, r): (t_1, t_2) \to \pi \times (0, \infty)$ is continuous. Indeed, suppose that $t_n \to t$. It is easy to see that $z$ and $r$ are bounded, and therefore we can pick subsequence $(t_{n_k})_{k \in \mathbb{N}}$ such that sequences $(z(t_{n_k}))_{k \in \mathbb{N}}$ and $(r(t_{n_k}))_{k \in \mathbb{N}}$ converge to some $z'$ and $r'$ respectively. It follows from the continuity of the norm that
$$||z'+r'a+tv||=||z'+r'b+tv||=||z'+r'c+tv||=1,$$
and taking the uniqueness into account we conclude that $z'=z(t)$ and $r'=r(t)$. This proves our claim.

In particular, mapping $r$ is continuous. Let $s \in (t_1, t_2)$ be such number that $B_s$ is a section containing $0$. Then $B_s$ is a smooth, strictly convex and symmetric convex body obtained by restricting the norm $|| \cdot ||$ to the two dimensional vector space parallel to $\pi$. By Theorem \ref{tw2}, it follows that $r(s) \geq 1$. Moreover, if $t \to t_1$, then $\diam B_t \to 0$. Therefore continuity implies that $r(t) = 1$ for some $t \in (t_1, t_2)$. Then
$$||z(t) + tv + a|| = ||z(t) + tv + b|| = ||z(t) + tv + c|| = 1,$$
and hence $d=-(z(t) + tv)$ is the desired point. It completes the proof in the case of a smooth and strictly convex norm.

Now let $|| \cdot ||$ be an arbitrary norm in $\mathbb{R}^3$. We shall reduce this case to the previous one by application of Lemma \ref{lem3} and Proposition \ref{aproksymacja}. As the points $a-b, b-c, c-a, b-a, c-b, a-c$ are vertices of the symmetric convex hexagon, Proposition \ref{aproksymacja} implies that for every $n \in \mathbb{N}$ we can find a smooth and strictly convex norm $|| \cdot ||_n$ such that 
$$\left (1-\frac{1}{n} \right ) ||x||_n \leq ||x|| \leq \left ( 1+\frac{1}{n} \right )||x||_n$$
for every $x \in \mathbb{R}^3$ and
$$||a-b||_n=||b-c||_n=||c-a||_n=1.$$

We have already proved, that for every $n \in \mathbb{N}$ we can find $d_n \in \mathbb{R}^3$ such that
$$||d_n - a ||_n = ||d_n - b||_n = ||d_n - c||_n = 1.$$
It is clear that the sequence $(d_n)_{n \in \mathbb{N}}$ is bounded and hence, it contains some subsequence $(d_{n_k})_{k \in \mathbb{N}}$ convergent to a point $d \in \mathbb{R}^3$. 
From the inequalities
$$ 1-\frac{1}{n_k} \leq ||d_{n_k}-a|| \leq 1+\frac{1}{n_k}$$
it follows that $||d-a|| = 1$ and analogously $||d-b||=||d-c||=1$. \qed

\section{Existence of a smooth and strictly convex norm with a $4$-element maximal equilateral set}
\label{sec:example}
In the space $\mathbb{R}^n$ we can introduce a norm $|| \cdot ||$ given by $||(x_1, x_2, \ldots, x_n)|| = |x_1| + \sqrt{x_2^2 + x_3^2 + \ldots + x_n^2}$. Petty has proved that for $n \geq 4$ it is possible to find a $4$-element equilateral set in such normed space, which is maximal with respect to inclusion. In \cite{swanepoel3} Swanepoel and Villa have generalized this example to every space of the form $X \oplus_1 \mathbb{R}$, where $X$ has at least one smooth point on the unit sphere. Spaces arising in such way are never smooth or strictly convex, however. In the same paper the authors have also proved that some of the $\ell^{n}_p$ spaces (with $n \geq 4$ and $1 < p < 2$) contain $5$-element equilateral sets which are maximal with respect to inclusion. It remains to answer, if a smooth and strictly convex space of dimension $n \geq 4$ can possess such a $4$-element equilateral set. Using the Proposition \ref{aproksymacja} we are ready to obtain, in a non-constructive way, the existence of a space with such property.

\paragraph{Proof of Theorem \ref{tw3}.} In the space $\mathbb{R}^n$ consider the $\ell_1$ norm (which of course is not smooth and not strictly convex itself) i.e. $||(x_1, x_2, \ldots, x_n)|| = |x_1| + |x_2| + \ldots + |x_n|$. Let
$$a_1 =  \bigg ( 1, 0, \ldots, 0  \bigg ), a_2 = \bigg ( -1, 0, \ldots, 0 \bigg ), a_3=\left ( 0, \frac{1}{n-1}, \frac{1}{n-1}, \frac{1}{n-1}, \ldots, \frac{1}{n-1} \right ),$$
$$a_4= \left(0,  -\frac{1}{4(n-1)}, -\frac{3}{4(n-1)}, -\frac{1}{n-1}, -\frac{1}{n-1}, \ldots, -\frac{1}{n-1} \right ).$$

It is immediate to check that $a_i$'s form a $2$-equilateral set in the $\ell_1$ norm. We will prove that there does not exist $x \in \mathbb{R}^n$ which expands this set to a $5$-point equilateral set. Indeed, suppose that $x=(x_1, x_2, \ldots, x_n)$ is such point. From the equalities
$$||x-a_1|| = ||x+a_1|| = 2$$
we conclude that $x_1=0$ and $|x_2|+|x_3| + \ldots |x_n| = 1$. Combining this with
$$||x-a_2||=||x-a_3||=2,$$
we get
$$ 2  = |x_2| + |x_3| + \ldots + |x_n| + 1 =\left | x_2 - \frac{1}{n-1} \right | + \left | x_3 - \frac{1}{n-1} \right | + \ldots \left | x_n-\frac{1}{n-1} \right |$$
$$ = \left | x_2 + \frac{1}{4(n-1)} \right | + \left | x_3 + \frac{3}{4(n-1)} \right | + \left | x_4 + \frac{1}{n-1} \right | + \ldots + \left | x_n + \frac{1}{n-1} \right |.$$
We shall show that such system of equations does not have a solution. Indeed, the casual triangle inequality easily implies that
$$|x_2| + |x_3| + \ldots + |x_n| + 1 \geq \left | x_2 - \frac{1}{n-1} \right | + \left  | x_3 - \frac{1}{n-1} \right | + \ldots + \left | x_n-\frac{1}{n-1} \right |$$
and
$$|x_2| + |x_3| + \ldots + |x_n| + 1 \geq  \left | x_2 + \frac{1}{4(n-1)} \right | + \left | x_3 + \frac{3}{4(n-1)} \right |+ \left | x_4 + \frac{1}{n-1} \right |$$
$$+ \ldots + \left | x_n + \frac{1}{n-1} \right |.$$
Moreover, in the inequality $|a| + |b| \geq |a+b|$ equality holds exactly when $a$ and $b$ are of the same sign. Therefore, vector $(x_1, x_2, \ldots, x_n) = (0, 0, \ldots, 0)$ is the only possible solution of the considered system, but it is not of the norm $1$. This proves our claim.

It is not hard to check by hand the linear independence of the vectors $a_2 - a_1$, $a_3 - a_1$ and $a_4 - a_1$, which implies, that $a_1$, $a_2$, $a_3$, $a_4$ are vertices of non-degenerated tetrahedron. By Lemma \ref{lem3} we know that set $ \{ a_i - a_j : 1 \leq i \neq j \leq 4 \}$ is the set of vertices of its convex hull. Applying Proposition \ref{aproksymacja}, we can pick a sequence $|| \cdot ||_k$ of smooth and strictly convex norms such that
$$\left (1-\frac{1}{k} \right ) ||x||_k \leq ||x|| \leq \left ( 1+\frac{1}{k} \right )||x||_k$$
for every $x \in \mathbb{R}^n$ and $||a_i - a_j||_k = 2$ for $1 \leq i \neq j \leq 4$ and $k \in \mathbb{N}$. 

Now assume on the contrary, that our assertion is not true. It means, that in any smooth and strictly convex norm in the space $\mathbb{R}^n$ (where $n \geq 4$), every $4$-point equilateral set can be expanded to a $5$-point equilateral set. In particular, for every $k \in \mathbb{N}$ there exists $x_k \in \mathbb{R}^n$ such that
$$||x_k - a_1||_k = ||x_k - a_2||_k = ||x_k - a_3||_k = ||x_k - a_4||_k = 2.$$
As $(x_k)_{k \in \mathbb{N}}$ is bounded, it has a subsequence convergent to $x \in \mathbb{R}^n$. Now it is easy to see that $x$ expands $a_1$, $a_2$, $a_3$, $a_4$ to a $5$-point equilateral set in the $\ell_1$ norm. This is a contradiction with the previous part of reasoning and the conclusion follows. \qed

\section{Concluding remarks}
\label{sec:7}
As we already mentioned in the introduction, probably the most natural question in the field of equilateral sets which remains open is
\begin{prob}[see \cite{grunbaum}, \cite{petty}, \cite{schutte}]
\label{pytanie}
Does the inequality $e(X) \geq n+1$ hold, for every normed space $X$ of dimension $n$?
\end{prob}
It is reasonable to ask if the approach that led us to the alternative proof of Petty's theorem can be helpful in answering Question \ref{pytanie} also in higher dimensions. Using similar approach Makeev has proved the case $n=4$ in \cite{makeev2}. It is also the largest dimension, for which the answer to Question \ref{pytanie} is known to be affirmative. Reduction of the general case to the situation in which the norm is smooth and strictly convex can be done in the exactly same way as in the presented proof. We can therefore try to use Theorem \ref{nemeth}. On the other hand, Proposition \ref{jedynosc} cannot be generalized to higher dimensions and in consequence it seems that establishing a continuous behaviour of the spheres, circumscribed about the variable simplex, is a serious technical issue. However, the main difficulty lies probably in obtaining the higher-dimensional analogue of Theorem \ref{tw2}. We already know from the previous section that such analogue could hold only for certain equilateral sets, i.e. we would have to find an equilateral set with the sphere of small radius passing through its vertices. This requires a different idea than for the planar case.


\begin{thebibliography}{HD}




\normalsize
\baselineskip=17pt


\bibitem{fabian} M. Fabian, P. Habala., P. H\'ajek, V.M. Santaluc\'ia, J. Pelant, V. Zizler,  \emph{Functional Analysis and Infinite-Dimensional Geometry}, CMS Books in Mathematics, Springer-Verlag (2001).
\bibitem{ghomi} M. Ghomi, \emph{Optimal smoothing for convex polytopes}, Bull. London Math. Soc. 36 (2004), 483--492. 
\bibitem{goodey} P.R. Goodey, \emph{Homothetic ellipsoids}, Math. Proc. Camb. Phil. Soc., 93 (1983), 25--34. 
\bibitem{gromov} M.L. Gromov, \emph{Simplexes inscribed in a hypersurface}, Mat. Zametki 5 (1969), 81--89. 
\bibitem{grunbaum} B. Gr\"unbaum, \emph{On a conjecture of H. Hadwiger},  Pacific J. Math. 11 (1961), 215--219.
\bibitem{kramer} H. Kramer, A.B. N\'emeth, \emph{The application of Brouwer's fixed point theorem to the geometry of convex bodies}, An. Univ. Timisoara Ser. Sti. Mat. 13 (1975), 33--39.
\bibitem{makeev1} V.V. Makeev, \emph{The degree of a mapping in some problems of combinatorial geometry}, J. Soviet Math. 51 (1987), 2544--2546.
\bibitem{makeev2} V.V. Makeev, \emph{Equilateral simplices in normed 4-space}, J. Math. Sci. (N. Y.) 140 (2007), 548--550.
\bibitem{geometry} H. Martin, K.J. Swanepoel, G. Weiss, \emph{The geometry of Minkowski spaces - a survey. Part I}, Expo. Math. 19 (2001), 97--142. Erratum Expo. Math. 19 (2001), 364.
\bibitem{martini} H. Martini, M. Spirova, \emph{Covering discs in Minkowski planes}, Canad. Math. Bull. 52 (2009), 424--434.
\bibitem{petty} C.M. Petty, \emph{Equilateral sets in Minkowski spaces}, Proc. Amer. Math. Soc. 29 (1971), 369--374.
\bibitem{schutte} K. Sch\"utte, \emph{Minimale Durchmesser endlicher Punktmengen mit vorgeschriebenem Mindestabstand}, Math. Ann. 150 (1963), 91--98.
\bibitem{soltan} P.S. Soltan, \emph{Analogues of regular simplexes in normed spaces}, Dokl. Akad. Nauk SSSR 222, 1303-1305 (1975). English translation: Soviet Math. Dokl. 16 (1975), 787--789.
\bibitem{swanepoel1} K.J. Swanepoel, \emph{Equilateral sets in finite-dimensional normed spaces}, in: Seminar of Mathematical Analysis, Daniel Girela \'Alvarez, Genaro L\'opez Acedo, Rafael Villa Caro (eds.), Secretariado de Publicationes, Universidad de Sevilla, Seville (2004), 195--237.
\bibitem{swanepoel2} K.J. Swanepoel, R. Villa, \emph{A lower bound for the equilateral number of normed spaces}, Proc. Amer. Math. Soc. 136 (2008), 127--131.
\bibitem{swanepoel3} K.J. Swanepoel, R. Villa, \emph{Maximal equilateral sets}, preprint, arXiv:1109.5063v1 [math.MG].
\end{thebibliography}
\end{document}